\documentclass[notitlepage,leqno,10pt]{amsart}
\usepackage{amsmath,amssymb,amsbsy,amsfonts,amsthm,latexsym,
            amsopn,amstext,amsxtra,euscript,amscd}
\usepackage{hyperref}

\begin{document}
\bibliographystyle{plain}
\newfont{\teneufm}{eufm10}
\newfont{\seveneufm}{eufm7}
\newfont{\fiveeufm}{eufm5}
%
%
\newfam\eufmfam
              \textfont\eufmfam=\teneufm \scriptfont\eufmfam=\seveneufm
              \scriptscriptfont\eufmfam=\fiveeufm
\def\bbbr{{\rm I\!R}}
\def\bbbm{{\rm I\!M}}
\def\bbbn{{\rm I\!N}}
\def\bbbf{{\rm I\!F}}
\def\bbbh{{\rm I\!H}}
\def\bbbk{{\rm I\!K}}
\def\bbbp{{\rm I\!P}}
\def\bbbone{{\mathchoice {\rm 1\mskip-4mu l} {\rm 1\mskip-4mu l}
{\rm 1\mskip-4.5mu l} {\rm 1\mskip-5mu l}}}
\def\bbbc{{\mathchoice {\setbox0=\hbox{$\displaystyle\rm C$}\hbox{\hbox
to0pt{\kern0.4\wd0\vrule height0.9\ht0\hss}\box0}}
{\setbox0=\hbox{$\textstyle\rm C$}\hbox{\hbox
to0pt{\kern0.4\wd0\vrule height0.9\ht0\hss}\box0}}
{\setbox0=\hbox{$\scriptstyle\rm C$}\hbox{\hbox
to0pt{\kern0.4\wd0\vrule height0.9\ht0\hss}\box0}}
{\setbox0=\hbox{$\scriptscriptstyle\rm C$}\hbox{\hbox
to0pt{\kern0.4\wd0\vrule height0.9\ht0\hss}\box0}}}}
\def\bbbq{{\mathchoice {\setbox0=\hbox{$\displaystyle\rm
Q$}\hbox{\raise
0.15\ht0\hbox to0pt{\kern0.4\wd0\vrule height0.8\ht0\hss}\box0}}
{\setbox0=\hbox{$\textstyle\rm Q$}\hbox{\raise
0.15\ht0\hbox to0pt{\kern0.4\wd0\vrule height0.8\ht0\hss}\box0}}
{\setbox0=\hbox{$\scriptstyle\rm Q$}\hbox{\raise
0.15\ht0\hbox to0pt{\kern0.4\wd0\vrule height0.7\ht0\hss}\box0}}
{\setbox0=\hbox{$\scriptscriptstyle\rm Q$}\hbox{\raise
0.15\ht0\hbox to0pt{\kern0.4\wd0\vrule height0.7\ht0\hss}\box0}}}}
\def\bbbt{{\mathchoice {\setbox0=\hbox{$\displaystyle\rm.
T$}\hbox{\hbox to0pt{\kern0.3\wd0\vrule height0.9\ht0\hss}\box0}}
{\setbox0=\hbox{$\textstyle\rm T$}\hbox{\hbox
to0pt{\kern0.3\wd0\vrule height0.9\ht0\hss}\box0}}
{\setbox0=\hbox{$\scriptstyle\rm T$}\hbox{\hbox
to0pt{\kern0.3\wd0\vrule height0.9\ht0\hss}\box0}}
{\setbox0=\hbox{$\scriptscriptstyle\rm T$}\hbox{\hbox
to0pt{\kern0.3\wd0\vrule height0.9\ht0\hss}\box0}}}}
\def\bbbs{{\mathchoice
{\setbox0=\hbox{$\displaystyle     \rm S$}\hbox{\raise0.5\ht0\hbox
to0pt{\kern0.35\wd0\vrule height0.45\ht0\hss}\hbox
to0pt{\kern0.55\wd0\vrule height0.5\ht0\hss}\box0}}
{\setbox0=\hbox{$\textstyle        \rm S$}\hbox{\raise0.5\ht0\hbox
to0pt{\kern0.35\wd0\vrule height0.45\ht0\hss}\hbox
to0pt{\kern0.55\wd0\vrule height0.5\ht0\hss}\box0}}
{\setbox0=\hbox{$\scriptstyle      \rm S$}\hbox{\raise0.5\ht0\hbox
to0pt{\kern0.35\wd0\vrule height0.45\ht0\hss}\raise0.05\ht0\hbox
to0pt{\kern0.5\wd0\vrule height0.45\ht0\hss}\box0}}
{\setbox0=\hbox{$\scriptscriptstyle\rm S$}\hbox{\raise0.5\ht0\hbox
to0pt{\kern0.4\wd0\vrule height0.45\ht0\hss}\raise0.05\ht0\hbox
to0pt{\kern0.55\wd0\vrule height0.45\ht0\hss}\box0}}}}
\def\bbbz{{\mathchoice {\hbox{$\sf\textstyle Z\kern-0.4em Z$}}
{\hbox{$\sf\textstyle Z\kern-0.4em Z$}}
{\hbox{$\sf\scriptstyle Z\kern-0.3em Z$}}
{\hbox{$\sf\scriptscriptstyle Z\kern-0.2em Z$}}}}
\def\ts{\thinspace}

\newtheorem{theorem}{Theorem}
\newtheorem{lemma}[theorem]{Lemma}
\newtheorem{claim}[theorem]{Claim}
\newtheorem{cor}[theorem]{Corollary}
\newtheorem{prop}[theorem]{Proposition}
\newtheorem{definition}[theorem]{Definition}
\newtheorem{remark}[theorem]{Remark}
\newtheorem{question}[theorem]{Open Question}
\newtheorem{example}[theorem]{Example}

\def\qed{\ifmmode
\squareforqed\else{\unskip\nobreak\hfil
\penalty50\hskip1em\null\nobreak\hfil\squareforqed
\parfillskip=0pt\finalhyphendemerits=0\endgraf}\fi}

\def\squareforqed{\hbox{\rlap{$\sqcap$}$\sqcup$}}

\def \C {{\mathbb C}}
\def \F {{\mathbb F}}
\def \L {{\mathbb L}}
\def \K {{\mathbb K}}
\def \Q {{\mathbb Q}}
\def \Z {{\mathbb Z}}
\def\cA{{\mathcal A}}
\def\cB{{\mathcal B}}
\def\cC{{\mathcal C}}
\def\cD{{\mathcal D}}
\def\cE{{\mathcal E}}
\def\cF{{\mathcal F}}
\def\cG{{\mathcal G}}
\def\cH{{\mathcal H}}
\def\cI{{\mathcal I}}
\def\cJ{{\mathcal J}}
\def\cK{{\mathcal K}}
\def\cL{{\mathcal L}}
\def\cM{{\mathcal M}}
\def\cN{{\mathcal N}}
\def\cO{{\mathcal O}}
\def\cP{{\mathcal P}}
\def\cQ{{\mathcal Q}}
\def\cR{{\mathcal R}}
\def\cS{{\mathcal S}}
\def\cT{{\mathcal T}}
\def\cU{{\mathcal U}}
\def\cV{{\mathcal V}}
\def\cW{{\mathcal W}}
\def\cX{{\mathcal X}}
\def\cY{{\mathcal Y}}
\def\cZ{{\mathcal Z}}
\newcommand{\rmod}[1]{\: \mbox{mod}\: #1}

\def\tcN{\cN^\mathbf{c}}
\def\F{\mathbb F}
\def\Tr{\operatorname{Tr}}
\def\mand{\qquad \mbox{and} \qquad}
\renewcommand{\vec}[1]{\mathbf{#1}}
\def\eqref#1{(\ref{#1})}
\newcommand{\ignore}[1]{}
\hyphenation{re-pub-lished}
\parskip 1.5 mm
\def\lln{{\mathrm Lnln}}
\def\Res{\mathrm{Res}\,}
\def\F{{\bbbf}}
\def\Fp{\F_p}
\def\fp{\Fp^*}
\def\Fq{\F_q}
\def\ff{\F_2}
\def\ffn{\F_{2^n}}
\def\K{{\bbbk}}
\def \Z{{\bbbz}}
\def \N{{\bbbn}}
\def\Q{{\bbbq}}
\def \R{{\bbbr}}
\def \P{{\bbbp}}
\def\Zm{\Z_m}
\def \Um{{\mathcal U}_m}
\def \Bf{\frak B}
\def\Km{\cK_\mu}
\def\va {{\mathbf a}}
\def \vb {{\mathbf b}}
\def \vc {{\mathbf c}}
\def\vx{{\mathbf x}}
\def \vr {{\mathbf r}}
\def \vv {{\mathbf v}}
\def\vu{{\mathbf u}}
\def \vw{{\mathbf w}}
\def \vz {{\mathbfz}}
\def\\{\cr}
\def\({\left(}
\def\){\right)}
\def\fl#1{\left\lfloor#1\right\rfloor}
\def\rf#1{\left\lceil#1\right\rceil}
\def\flq#1{{\left\lfloor#1\right\rfloor}_q}
\def\flp#1{{\left\lfloor#1\right\rfloor}_p}
\def\flm#1{{\left\lfloor#1\right\rfloor}_m}
\def\Al{{\sl Alice}}
\def\Bob{{\sl Bob}}
\def\Or{{\mathcal O}}
\def\inv#1{\mbox{\rm{inv}}\,#1}
\def\invM#1{\mbox{\rm{inv}}_M\,#1}
\def\invp#1{\mbox{\rm{inv}}_p\,#1}
\def\Ln#1{\mbox{\rm{Ln}}\,#1}
\def \nd {\,|\hspace{-1.2mm}/\,}
\def\ord{\mu}
\def\E{\mathbf{E}}
\def\Cl{{\mathrm {Cl}}}
\def\epp{\mbox{\bf{e}}_{p-1}}
\def\ep{\mbox{\bf{e}}_p}
\def\eq{\mbox{\bf{e}}_q}
\def\bm{\bf{m}}
\newcommand{\floor}[1]{\lfloor {#1} \rfloor}
\newcommand{\comm}[1]{\marginpar{
\vskip-\baselineskip
\raggedright\footnotesize
\itshape\hrule\smallskip#1\par\smallskip\hrule}}
\def\rem{{\mathrm{\,rem\,}}}
\def\dist {{\mathrm{\,dist\,}}}
\def\etal{{\it et al.}}
\def\ie{{\it i.e. }}
\def\veps{{\varepsilon}}
\def\eps{{\eta}}
\def\ind#1{{\mathrm {ind}}\,#1}
               \def \MSB{{\mathrm{MSB}}}
\newcommand{\abs}[1]{\left| #1 \right|}

\title{Pell numbers whose Euler function is a Pell number}
%

\author{
{\sc Bernadette~Faye}\quad and \quad
{\sc Florian~Luca}
}

\address{
Ecole Doctorale de Mathematiques et d'Informatique \newline
Universit\'e Cheikh Anta Diop de Dakar \newline
BP 5005, Dakar Fann, Senegal and \newline 
School of Mathematics\newline 
University of the Witwatersrand \newline
Private Bag X3, Wits 2050, South Africa
}
\email{bernadette@aims-senegal.org}
\address{
School of Mathematics, University of the Witwatersrand \newline
Private Bag X3, Wits 2050, South Africa \newline
}
\email{Florian.Luca@wits.ac.za}
\pagenumbering{arabic}

\begin{abstract} In this paper, we show that the only Pell numbers whose Euler function is also a Pell number are $1$ and $2$.
\end{abstract}

\maketitle

\par\noindent {\small{\bf AMS Subject Classification 2010}}: {\small Primary 11B39; Secondary 11A25}

\par\noindent {\small{\bf Keywords}}: {\small Pell numbers; Euler function; Applications of sieve methods.}

\section{Introduction}

Let $\phi(n)$ be the Euler function of the positive integer $n$. Recall that if $n$ has the prime factorization
$$
n=p_1^{a_1}\cdots p_k^{a_k}
$$
with distinct primes $p_1,\ldots,p_k$ and positive integers $a_1,\ldots,a_k$, then
$$
\phi(n)=p_1^{a_1-1}(p_1-1)\cdots p_k^{a_k-1}(p_k-1).
$$
There are many papers in the literature dealing with diophantine equations
involving the Euler function in members of a binary recurrent sequence. For example, in \cite{FN}, it is shown that $1,~2$, and $3$ are the only Fibonacci numbers whose Euler function is also a Fibonacci number, while in \cite{FLT} it is shown that the Diophantine equation $\phi(5^n-1)=5^m-1$ has no positive integer solutions $(m,n)$. Furthermore, the divisibility relation $\phi(n)\mid n-1$ when $n$ is a Fibonacci number, or a Lucas number, or a Cullen number (that is, a number of the form $n2^n+1$ for some positive integer $n$), or a rep-digit
$(g^m-1)/(g-1)$ in some integer base $g\in [2,1000]$ have been investigated in \cite{L}, \cite{FL}, \cite{GL} and \cite{CL}, respectively.

Here we look at a similar equation with members of the {\it Pell sequence}. The Pell sequence $(P_n)_{n\ge 0}$ is given by $P_0=0$, $P_1=1$ and $P_{n+1}=2P_{n}+ P_{n-1}$ for all $n\geq0$. Its first terms are
$$
0,1,2,5,12,29,70,169,408,985,2378,5741,13860,33461,80782,195025,470832,\ldots$$
We have the following result.
\begin{theorem}
The only solutions in positive integers $(n,m)$ of the equation
\begin{equation}
\label{eq:pb}
\phi(P_n)=P_m
\end{equation}
are $(n,m)=(1,1),(2,1).$
\end{theorem}
For the proof, we begin by following the method from \cite{FN}, but we add to it some ingredients from \cite{L}.

\section{Preliminary results}

Let $(\alpha,\beta)=(1+{\sqrt{2}},1-{\sqrt{2}})$ be the roots of the characteristic equation $x^2-2x-1=0$ of the Pell sequence $\{P_n\}_{n\ge 0}$. The Binet formula for $P_n$ is
\begin{equation}
\label{eq:BinetP}
P_n= \frac{\alpha^n - \beta^n}{\alpha-\beta} \quad {\text{\rm for~ all}}\quad  n\ge 0.
\end{equation}
This implies easily that the inequalities
\begin{equation}
\label{eq:sizePn}
\alpha^{n-2}\le P_n\le  \alpha^{n-1}
\end{equation}
hold for all positive integers $n$.

We let $\{Q_n\}_{n\geq 0}$ be the companion Lucas sequence of the Pell sequence given by $Q_0=2$, $Q_1=2$ and $Q_{n+2}=2Q_{n+1}+ Q_{n}$ for all $n\ge 0$. Its first few terms are
$$
2, 2, 6, 14, 34, 82, 198, 478, 1154, 2786, 6726, 16238, 39202, 94642,228486,551614,\ldots
$$
The Binet formula for $Q_n$ is
\begin{equation}
\label{eq:BinetQ}
Q_n= \alpha^n + \beta^n\quad {\text{\rm for~ all}}\quad n\ge 0.
\end{equation}
We use the well-known result.
\begin{lemma}
\label{lem:PQ}
The relations
\begin{itemize}
\item[(i)] $P_{2n}=P_n Q_n$,
\item[(ii)] $Q_n^2 - 8P_n^2=4(-1)^n$
\end{itemize}
hold for all $n\ge 0$.
\end{lemma}
For a prime $p$ and a nonzero integer $m$ let $\nu_p(m)$ be the exponent with which $p$ appears in the prime factorization of $m$.  The following result is well-known and easy to prove.
\begin{lemma}
\label{lem:orderof2}
The relations
\begin{itemize}
\item[(i)] $\nu_2(Q_n)=1$,
\item[(ii)] $\nu_2(P_n)=\nu_2(n)$
\end{itemize}
hold for all positive integers $n$.
\end{lemma}
The following divisibility relations among the Pell numbers are well-known.
\begin{lemma}
\label{lem:div}
Let $m$ and $n$ be positive integers. We have:
\begin{itemize}
\item[(i)] If $m\mid n$ then $P_m\mid P_n$,
\item[(ii)] $\gcd(P_m,P_n)=P_{\gcd(m,n)}$.
\end{itemize}
\end{lemma}
For each positive integer $n$, let $z(n)$ be the smallest positive integer $k$ such that $n\mid P_k.$ It is known that this exists and $n\mid P_m$ if and only if $z(n)\mid m$. This number is referred to
as \textit{the order of appearance of $n$} in the Pell sequence. Clearly, $z(2)=2$. Further, putting for an odd prime $p$, $e_p={\displaystyle{\left(\frac{2}{p}\right)}}$, where the above notation stands for the Legendre symbol of $2$ with respect to $p$, we have that $z(p)\mid p-e_p$. A prime factor $p$ of $P_n$ such that $z(p)=n$ is called
\textit{primitive for $P_n$}. It is known that $P_n$ has a primitive divisor for all $n\ge 2$ (see \cite{Car} or \cite{BHV}). Write $P_{z(p)}=p^{e_p} m_p$, where $m_p$ is coprime to $p$. It is known that if $p^k\mid P_n$ for some $k>e_p$, then $pz(p)\mid n$. In particular,
\begin{equation}
\label{eq:nupvaluationofPn}
\nu_p(P_n)\le e_p\quad {\text{\rm whenever}}\quad p\nmid n.
\end{equation}
We need a bound on $e_p$. We have the following result.

\begin{lemma}
\label{lem:zofp}
The inequality
\begin{equation}
\label{eq:eofp}
e_p\le \frac{(p+1)\log \alpha}{2\log p}.
\end{equation}
holds for all primes $p$.
\end{lemma}

\begin{proof}
Since $e_2=1$, the inequality holds for the prime $2$. Assume that $p$ is odd. Then $z(p)\mid p+\varepsilon$ for some $\varepsilon\in \{\pm 1\}$. Furthermore, by Lemmas \ref{lem:PQ} and \ref{lem:div}, we have
$$
p^{e_p}\mid P_{z(p)}\mid P_{p+\varepsilon}=P_{(p+\varepsilon)/2}Q_{(p+\varepsilon)/2}.
$$
By Lemma \ref{lem:PQ}, it follows easily that $p$ cannot divide both $P_n$ and $Q_n$ for $n=(p+\varepsilon)/2$ since otherwise $p$ will also divide
$$
Q_n^2-8P_n^2=\pm 4,
$$
a contradiction since $p$ is odd. Hence, $p^{e_p}$ divides one of $P_{(p+\varepsilon)/2}$ or $Q_{(p+\varepsilon)/2}$. If $p^{e_p}$ divides $P_{(p+\varepsilon)/2}$, we have, by \eqref{eq:sizePn}, that
$$
p^{e_p}\le P_{(p+\varepsilon)/2}\le P_{(p+1)/2}<\alpha^{(p+1)/2},
$$
which leads to the desired inequality \eqref{eq:eofp} upon taking logarithms of both sides. In case $p^{e_p}$ divides $Q_{(p+\varepsilon)/2}$, we use the fact that $Q_{(p+\varepsilon)/2}$ is even by Lemma \ref{lem:orderof2} (i). Hence, $p^{e_p}$ divides $Q_{(p+\varepsilon)/2}/2$, therefore, by formula \eqref{eq:BinetQ}, we have
$$
p^{e_p}\le \frac{Q_{(p+\varepsilon)/2}}{2}\le \frac{Q_{(p+1)/2}}{2}<\frac{\alpha^{(p+1)/2}+1}{2}<\alpha^{(p+1)/2},
$$
which leads again to the desired conclusion by taking logarithms of both sides.
\end{proof}

For a positive real number $x$ we use $\log x$ for the natural logarithm of $x$. We need some inequalities from the prime number theory. For a positive integer $n$ we write $\omega(n)$ for the number of distinct prime factors of $n$. The following inequalities (i), (ii) and (iii) are inequalities (3.13), (3.29) and (3.41) in \cite{RS}, while (iv) is Th\'eor\'eme 13 from \cite{MNR}.

\begin{lemma}
\label{lem:RS}
Let $p_1<p_2<\cdots$ be the sequence of all prime numbers. We have:
\begin{itemize}
\item[(i)] The inequality
$$
p_n<n(\log n+\log\log n)
$$
holds for all $n\ge 6$.
\item[(ii)]
The inequality
$$
\prod_{p\le x} \left(1+\frac{1}{p-1}\right)<1.79\log x\left(1+\frac{1}{2(\log x)^2}\right)
$$
holds for all $x\ge 286$.
\item[(iii)]
The inequality
$$
\phi(n)>\frac{n}{1.79 \log\log n+2.5/\log\log n}
$$
holds for all $n\ge 3$.
\item[(iv)] The inequality
$$
\omega(n)<\frac{\log n}{\log \log n-1.1714}
$$
holds for all $n\ge 26$.
\end{itemize}
\end{lemma}
For a positive integer $n$, we put ${\mathcal P}_n=\{p: z(p)=n\}$.
We need the following result.
\begin{lemma}
\label{lem:sumSn}
Put
$$
S_n:=\sum_{p\in {\mathcal P}_n} \frac{1}{p-1}.
$$
For $n>2$, we have
\begin{equation}
\label{eq:Sn}
S_n<\min\left\{\frac{2\log n}{n},\frac{4+4\log\log n}{\phi(n)}\right\}.
\end{equation}
\end{lemma}

\begin{proof}
Since $n>2$, it follows that every prime factor $p\in {\mathcal P}_n$ is odd and satisfies the congruence $p\equiv \pm 1\pmod n$. Further, putting $\ell_n:=\#{\mathcal P}_n$, we have
$$
(n-1)^{\ell_n}\le \prod_{p\in {\mathcal P}_n} p\le P_n<\alpha^{n-1}
$$
(by inequality \eqref{eq:sizePn}), giving
\begin{equation}
\label{eq:S}
\ell_n\le \frac{(n-1)\log \alpha}{\log(n-1)}.
\end{equation}
Thus, the inequality
\begin{equation}
\label{eq:ellofn}
\ell_n<\frac{n\log \alpha}{\log n}
\end{equation}
holds for all $n\ge 3$, since it follows from \eqref{eq:S} for $n\ge 4$ via the fact that the function $x\mapsto x/\log x$ is increasing for $x\ge 3$, while for $n=3$ it can be checked directly. To prove the first bound,
we use \eqref{eq:ellofn} to deduce that
\begin{eqnarray}
\label{eq:last}
S_n & \le & \sum_{1\le \ell\le \ell_n} \left(\frac{1}{n\ell-2} +\frac{1}{n\ell}\right)\nonumber\\
& \le &  \frac{2}{n}\sum_{1\le \ell \le \ell_n} \frac{1}{\ell}+\sum_{m\ge n} \left(\frac{1}{m-2}-\frac{1}{m}\right)\nonumber\\
& \le & \frac{2}{n} \left(\int_{1}^{\ell_n} \frac{dt}{t} +1\right)+\frac{1}{n-2}+\frac{1}{n-1}\nonumber\\
& \le & \frac{2}{n}\left(\log \ell_n+1+\frac{n}{n-2}\right)\nonumber\\
& \le & \frac{2}{n} \log\left( n\left(\frac{(\log \alpha) e^{2+2/(n-2)}}{\log n}\right)\right).
\end{eqnarray}
Since the inequality
$$
\log n>(\log \alpha) e^{2+2/(n-2)}
$$
holds for all $n\ge 800$, \eqref{eq:last} implies that
$$
S_n<\frac{2\log n}{n}\quad {\text{\rm for}}\quad n\ge 800.
$$
The remaining range for $n$ can be checked on an individual basis.
For the second bound on $S_n$, we follow the argument from \cite{L} and split the primes in ${\mathcal P}_n$ in three groups:
\begin{itemize}
\item[(i)] $p<3n$;
\item[(ii)] $p\in (3n,n^2)$;
\item[(iii)] $p>n^2$;
\end{itemize}
We have
\begin{equation}
\label{eq:T1}
T_1=\sum_{\substack{p\in {\mathcal P}_n\\ p<3n}} \frac{1}{p-1}
\le \left\{
\begin{matrix}
{\displaystyle{\frac{1}{n-2}+\frac{1}{n}+\frac{1}{2n-2}+\frac{1}{2n}+\frac{1}{3n-2} }} & < &
{\displaystyle{\frac{10.1}{3n}}},& n\equiv 0\pmod 2,\\
{\displaystyle{\frac{1}{2n-2}+\frac{1}{2n}}} & < & {\displaystyle{\frac{7.1}{3n}}}, & n\equiv 1\pmod 2,\\
\end{matrix}\right.
\end{equation}
where the last inequalities above hold for all $n\ge 84$. For the remaining primes in ${\mathcal P}_n$, we have
\begin{equation}
\label{eq:T2T3}
\sum_{\substack{p\in {\mathcal P}_n\\ p>3n}} \frac{1}{p-1}<
\sum_{\substack{p\in {\mathcal P}_n\\ p>3n}} \frac{1}{p}+\sum_{m\ge 3n+1} \left(\frac{1}{m-1}-\frac{1}{m}\right)=T_2+T_3+\frac{1}{3n},
\end{equation}
where $T_2$ and $T_3$ denote the sums of the reciprocals of the primes in ${\mathcal P}_n$ satisfying (ii) and (iii), respectively. The sum $T_2$ was estimated in \cite{L} using the large sieve inequality of Montgomery and Vaughan \cite{MV} (see also page 397 in \cite{FN}), and the bound on it is
\begin{equation}
\label{eq:T2}
T_2=\sum_{3n<p<n^2} \frac{1}{p}<\frac{4}{\phi(n)\log n}+\frac{4\log\log n}{\phi(n)}<\frac{1}{\phi(n)}+\frac{4\log\log n}{\phi(n)},
\end{equation}
where the last inequality holds for $n\ge 55$. Finally, for $T_3$, we use the estimate \eqref{eq:ellofn} on $\ell_n$ to deduce that
\begin{equation}
\label{eq:T3}
T_3< \frac{\ell_n}{n^2}<\frac{\log \alpha}{n\log n}<\frac{0.9}{3n},
\end{equation}
where the last bound holds for all $n\ge 19$. To summarize, for $n\ge 84$, we have, by \eqref{eq:T1}, \eqref{eq:T2T3}, \eqref{eq:T2} and \eqref{eq:T3},
$$
S_n<\frac{10.1}{3n}+\frac{1}{3n}+\frac{0.9}{3n}+\frac{1}{\phi(n)}+\frac{4\log\log n}{\phi(n)}=\frac{4}{n}+\frac{1}{\phi(n)}+\frac{4\log\log n}{\phi(n)}\le \frac{3+4\log\log n}{\phi(n)}
$$
for $n$ even, which is stronger that the desired inequality. Here, we used that $\phi(n)\le n/2$ for even $n$. For odd $n$, we use the same argument except that the first fraction $10.1/(3n)$ on the right--hand side above gets replaced by $7.1/(3n)$ (by \eqref{eq:T1}), and we only have $\phi(n)\le n$ for odd $n$.
This was for $n\ge 84$. For $n\in [3,83]$, the desired inequality can be checked on an individual basis.
\end{proof}

The next lemma from \cite{Lu2} gives an upper bound on the sum appearing in the right--hand side of \eqref{eq:Sn}.

\begin{lemma}
\label{lem:useful}
We have
$$
\sum_{d\mid n} \frac{\log d}{d}<\left(\sum_{p\mid n} \frac{\log p}{p-1}\right) \frac{n}{\phi(n)}.
$$
\end{lemma}

Throughout the rest of this paper we use $p,~q,~r$ with or without subscripts to denote prime numbers.

\section{Proof of The Theorem}

\subsection{Some lower bounds on $m$ and $\omega(P_n)$}

We start with a computation showing that there are no other solutions than 
$n=1,~2$ when $n\le 100$. So, from now on $n>100$. We write
\begin{equation}
\label{eq:1}
P_n=q_1^{\alpha_1}\ldots q_k^{\alpha_k},
\end{equation}
where $q_1<\cdots <q_k$  are primes and $\alpha_1,\ldots,\alpha_k$ are positive integers. Clearly, $m<n$.

McDaniel \cite{MD}, proved that $P_n$ has a prime factor $q\equiv 1\pmod 4$
for all $n>14$. Thus, McDaniel's result applies for us showing that
$$
4\mid q-1\mid \phi(P_n)\mid P_m,
$$
so $4\mid m$ by Lemma \ref{lem:orderof2}. Further,
it follows from a the result of the second author \cite{FL}, that $\phi(P_n)\geq P_{\phi(n)}.$ Hence, $m\ge \phi(n)$.
Thus,
\begin{equation}
\label{eq:low}
m\geq \phi(n)\geq \frac{n}{1.79\log\log n + 2.5/\log\log n},
\end{equation}
by Lemma \ref{lem:RS} (iii). The function
$$
x\mapsto \frac{x}{1.79 \log\log x+2.5/\log\log x}
$$
is increasing for $x\ge 100$. Since $n\ge 100$, inequality \eqref{eq:low} together with the fact that $4\mid m$, show that $m\ge 24$.

Put $\ell=n-m$. Since $m$ is even, we have $\beta^m>0$, therefore
\begin{equation}
\label{eq:10tominus40}
\frac{P_n}{P_m}=\frac{\alpha^n-\beta^n}{\alpha^m -\beta^m}> \frac{\alpha^n-\beta^n}{\alpha^m}\ge \alpha^\ell -\frac{1}{\alpha^{m+n}} > \alpha^\ell - 10^{-40},
\end{equation}
where we used the fact that
$$
\frac{1}{\alpha^{m+n}}\le \frac{1}{\alpha^{124}}<10^{-40}.
$$
We now are ready to provide a large lower bound on $n$. We distinguish the following cases.

\medskip

\textbf{\small{Case 1}}: {\it $n$ is odd}.

\medskip

Here, we have $\ell\geq 1$. So,
$$
\frac{P_n}{P_m}> \alpha - 10^{-40}>2.4142.
$$
Since $n$ is odd, it follows that $P_n$ is divisible only by primes $q$ such that $z(q)$ is odd. Among the first $10000$ primes, there are precisely $2907$ of them with this property. They are
$$
\mathcal{F}_1=\{5, 13, 29, 37,53, 61, 101, 109, \ldots,104597, 104677, 104693, 104701, 104717\}.
$$
Since
$$
\prod_{p\in {\mathcal F}_1}\left(1-\frac{1}{p}\right)^{-1}<1.963<2.4142<\frac{P_n}{P_m}=\prod_{i=1}^k \left(1-\frac{1}{q_i}\right)^{-1},
$$
we get that $k>2907$. Since $2^k\mid \phi(P_n)\mid P_m$, we get,
by Lemma \ref{lem:orderof2}, that
\begin{equation}
\label{eq:Case1}
n>m>2^{2907}.
\end{equation}

\textbf{\small{Case 2}}: $n\equiv 2\pmod 4$.

Since both $m$ and $n$ are even, we get $\ell\geq 2.$ Thus,
\begin{equation}
\label{eq:lowBoundCase2}
\frac{P_n}{P_m}> \alpha^2 - 10^{-40} > 5.8284.
\end{equation}
If $q$ is a prime factor of $P_n$, as in Case 1, we have that $z(q)$ is
not divisible by $4$.  Among the first $10000$ primes, there are precisely $5815$ of them with this property.
They are
$$
\mathcal{F}_2=\{2, 5, 7, 13,23, 29, 31, 37, 41,47, 53, 61, \ldots, 104693, 104701, 104711, 104717\}.
$$
Writing $p_j$ as the $j$th prime number in  $\mathcal{F}_2$, we check with Mathematica that
\begin{eqnarray*}
\prod_{i=1}^{415}\left(1-\frac{1}{p_i}\right)^{-1} & = & 5.82753\ldots\\
\prod_{i=1}^{416} \left(1-\frac{1}{p_i}\right)^{-1} & = & 5.82861\ldots,
\end{eqnarray*}
which via inequality \eqref{eq:lowBoundCase2} shows that $k\ge 416$. Of the $k$ prime factors of $P_n$, we have that only $k-1$ of them are odd ($q_1=2$ because $n$ is even), but one of those is congruent to $1$ modulo $4$ by McDaniel's result. Hence, $2^k\mid \phi(P_n)\mid P_m$, which shows, via Lemma \ref{lem:orderof2}, that
\begin{equation}
\label{eq:Case2}
n>m\ge 2^{416}.
\end{equation}

\textbf{\small{Case 3}}: $4 \mid n$.

In this case, since both $m$ and $n$ are multiples of $4$, we get that $\ell\geq 4$. Therefore,
$$
\frac{P_n}{P_m}> \alpha^4 - 10^{-40} > 33.97.
$$
Letting $p_1<p_2<\cdots$ be the sequence of all primes, we have that
$$
\prod_{i=1}^{2000}\left(1-\frac{1}{p_i}\right)^{-1}<17.41\ldots<33.97<\frac{P_n}{P_m}=
\prod_{i=1}^k \left(1-\frac{1}{q_i}\right),
$$
showing that $k>2000$. Since $2^{k}\mid \phi(P_n)=P_m$, we get
 \begin{equation}
\label{eq:Case3}
n> m\geq 2^{2000}.
\end{equation}
To summarize, from \eqref{eq:Case1}, \eqref{eq:Case2} and \eqref{eq:Case3}, we get the following results.

\begin{lemma}
\label{lem:1}
If $n>2$, then
\begin{enumerate}
\item $2^k\mid m$;
\item $k\ge 416$;
\item $n>m\ge 2^{416}$.
\end{enumerate}
\end{lemma}

\subsection{\small Bounding $\ell$ in term of $n$}

We saw in the preceding section that $k\geq 416$. Since $n>m\ge 2^k$, we have
\begin{equation}
\label{eq:k}
k<k(n):=\frac{\log n}{\log 2}.
\end{equation}
Let $p_j$ be the $j$th prime number. Lemma \ref{lem:RS} shows that
$$
p_k\le p_{\lfloor k(n)\rfloor}\le k(n)(\log k(n) +\log\log k(n)):=q(n).
$$
We then have, using Lemma \ref{lem:RS} (ii), that
$$
\frac{P_m}{P_n}=\prod_{i=1}^{k} \left(1-\frac{1}{q_i}\right)\geq \prod_{2\leq p\leq q(n)}\left(1-\frac{1}{p}\right)>\frac{1}{1.79\log q(n)(1+1/(2(\log q(n))^2))}.
$$
Inequality (ii) of Lemma \ref{lem:RS} requires that $x\ge 286$, which holds for us with $x=q(n)$ because $k(n)\ge 416$. Hence, we get
$$
1.79\log q(n)\left(1+  \frac{1}{(2(\log q(n))^2)}\right)>\frac{P_n}{P_m}>\alpha^{\ell}-10^{-40}>\alpha^{\ell}
\left(1-\frac{1}{10^{40}}\right).
$$
Since $k\ge 416$, we have $q(n)> 3256$. Hence, we get
$$\log q(n) \left(1.79\left(1-\frac{1}{10^{40}}\right)^{-1}
\left(1+\frac{1}{2(\log(3256))^2}\right)\right)>\alpha^\ell,
$$
which yields, after taking logarithms, to
\begin{equation}
\label{eq:5}
\ell \le \frac{\log\log q(n)}{\log \alpha}+0.67.
\end{equation}
The inequality
\begin{equation}
\label{eq:1.5}
q(n)<(\log n)^{1.45}
\end{equation}
holds in our range for $n$ (in fact, it holds for all $n>10^{83}$, which is our case since for us $n>2^{416}>10^{125}$). Inserting inequality \eqref{eq:1.5} into \eqref{eq:5}, we get
$$
\ell<\frac{\log\log (\log n)^{1.45}}{\log \alpha}+0.67<\frac{\log\log\log n}{\log \alpha}+1.1.
$$
Thus, we proved the following result.

\begin{lemma}
\label{lem:2}
If $n>2$, then
\begin{equation}
\ell < \frac{\log\log\log n}{\log \alpha}+1.1.
\end{equation}
\end{lemma}

\subsection{\small Bounding the primes $q_i$ for $i=1,\ldots,k$}

Write
\begin{equation}
\label{eq:51}
P_n= q_1\cdots q_k B,\quad {\text{\rm where}}\quad B=q_1^{\alpha_1-1} \cdots q_k^{\alpha_k-1}.
\end{equation}
Clearly, $B\mid \phi(P_n)$, therefore $B\mid P_m$. Since also $B\mid P_n$, we have, by Lemma \ref{lem:div}, that $B\mid \gcd(P_n,P_m)=P_{\gcd(n,m)}\mid P_{\ell}$ where the last relation follows again  by Lemma \ref{lem:div} because $\gcd(n,m)\mid \ell.$ Using the inequality \eqref{eq:sizePn} and Lemma \ref{lem:2}, we get
\begin{equation}
\label{eq:6}
B\leq P_{n-m}\leq \alpha^{n-m-1}\leq \alpha^{0.1}\log\log n.
\end{equation}
To bound the primes $q_i$ for all $i=1,\ldots,k$, we use the inductive argument from Section 3.3 in \cite{FN}. We write
$$
\prod_{i=1}^{k} \left(1-\frac{1}{q_i}\right)=\frac{\phi(P_n)}{P_n}=\frac{P_m}{P_n}.
$$
Therefore,
$$ 1-
\prod_{i=1}^{k} \left(1-\frac{1}{q_i}\right)=1-\frac{P_m}{P_n}=\frac{P_n-P_m}{P_n} \ge \frac{P_n-P_{n-1}}{P_n}>\frac{P_{n-1}}{P_n}.
$$
Using the inequality
\begin{equation}
\label{eq:7}
 1 -(1-x_1)\cdots(1-x_s) \leq x_1 + \cdots + x_s\quad {\text{\rm valid for all}}\quad x_i \in [0,1]~ {\text{\rm for}}~ i= 1,\ldots,s,
\end{equation}
we get,
$$
\frac{P_{n-1}}{P_n} < 1- \prod_{i=1}^{k} \(1-\frac{1}{q_i}\) \leq \sum_{i=1}^{k} \frac{1}{q_i} < \frac{k}{q_1},$$
therefore,
\begin{equation}
\label{eq:8}
q_1< k \left(\frac{P_{n}}{P_{n-1}}\right) < 3k.
\end{equation}
Using the method of the proof of inequality (13) in  \cite{FN}, one proves by induction on the index $i\in \{1,\ldots,k\}$ that if we put
$$
u_i:=\prod_{j=1}^{i} q_j,
$$
then
\begin{equation}
\label{eq:9}
u_i< \(2\alpha^{2.1} k \log\log n\)^{(3^i -1)/2}.
\end{equation}
In particular,
$$
q_1\cdots q_k = u_k <(2\alpha^{2.1}k\log\log n)^{(3^k - 1)/2},
$$
which together with formula \eqref{eq:5} and \eqref{eq:6} gives
$$
P_n=q_1\cdots q_k B<(2\alpha^{2.1}k\log\log n)^{1+(3^k-1)/2}= (2\alpha^{2.1}k\log\log n)^{(3^k+1)/2}.
$$
Since $P_n > \alpha^{n-2}$ by inequality \eqref{eq:sizePn}, we get
$$
(n-2)\log\alpha < \frac{(3^k+1)}{2}\log(2\alpha^{2.1}k\log\log n).
$$
Since $k<\log n/\log 2$ (see \eqref{eq:k}), we get
\begin{eqnarray*}
3^k & > & (n-2)\left(\frac{2\log\alpha}{\log(2\alpha^{2.1}(\log n)(\log\log n) (\log 2)^{-1})}\right)-1\\
& > & 0.17(n-2)-1>\frac{n}{6},
\end{eqnarray*}
where the last two inequalities above hold because $n>2^{416}$.

So, we proved the following result.

\begin{lemma}
\label{lem:3}
If $n>2$, then 
$$
3^k>n/6.
$$
\end{lemma}

\subsection{The case when $n$ is odd}

Assume that $n>2$ is odd and let $q$ be any prime factor of $P_n$. Reducing relation
\begin{equation}
\label{eq:nodd}
Q_n^2 - 8P_n^2 = 4 (-1)^n
\end{equation}
of Lemma \ref{lem:PQ} (ii) modulo $q$, we get $Q_n^2\equiv -4\pmod q$. Since $q$ is odd, (because $n$ is odd), we get that $q\equiv 1\pmod 4$. This is true for all prime factors
$q$ of $P_n$. Hence,
$$
4^k\mid \prod_{i=1}^k (q_i-1)\mid \phi(P_n)\mid P_m,
$$
which, by Lemma \ref{lem:orderof2} (ii), gives $4^k\mid m$. Thus,
$$
n>m\geq 4^{k},
$$
inequality which together with Lemma \ref{lem:3} gives
$$
n>\left(3^k\right)^{\log 4/\log 3}>\left(\frac{n}{6}\right)^{\log 4/\log 3},
$$
so
$$
n<6^{\log 4/\log(4/3)}<5621,
$$
in contradiction with Lemma \ref{lem:1}.

\subsection{\small Bounding $n$}

From now on, $n>2$ is even. We write it as
$$
n=2^s r_1^{\lambda_1}\cdots r_t^{\lambda_t}=:2^s n_1,
$$
where $s\ge 1$, $t\ge 0$ and $3\le r_1<\cdots<r_t$ are odd primes. Thus, by inequality \eqref{eq:10tominus40}, we have
\begin{eqnarray*}
\alpha^{\ell}\left(1-\frac{1}{10^{40}}\right) & < & \alpha^{\ell}-\frac{1}{10^{40}}<\frac{P_n}{\phi(P_n)}\\
& = & \prod_{p\mid P_n} \left(1+\frac{1}{p-1}\right)\\
& = &
2 \prod_{\substack{d\ge 3\\ d\mid n}} \prod_{p\in {\mathcal P}_d} \left(1+\frac{1}{p-1}\right),
\end{eqnarray*}
and taking logarithms we get
\begin{eqnarray}
\label{eq:veryuseful}
\ell\log \alpha-\frac{1}{10^{39}} & < & \log\left(\alpha^{\ell} \left(1-\frac{1}{10^{40}}\right)\right)\nonumber\\
& < & \log 2+\sum_{\substack{d\ge 3\\ d\mid n}} \sum_{p\in {\mathcal P}_d} \log\left(1+\frac{1}{p-1}\right)\nonumber\\
& < & \log 2+\sum_{\substack{d\ge 3\\ d\mid n}} S_d.
\end{eqnarray}
In the above, we used the inequality $\log(1-x)>-10 x$ valid for all $x\in (0,1/2)$ with $x=1/10^{40}$ and 
the inequality $\log(1+x)\le x$ valid for all real numbers $x$ with $x=p$ for all $p\in {\mathcal P}_d$ and all divisors $d\mid n$ with $d\ge 3$.

Let us deduce that the case $t=0$ is impossible. Indeed, if this were so, then $n$ is a power of $2$ and so, by Lemma \ref{lem:1}, both $m$ and $n$ are divisible by $2^{416}$. Thus, $\ell\ge 2^{416}$. Inserting this into \eqref{eq:veryuseful}, and using Lemma \ref{lem:sumSn}, we get
$$
2^{416}\log \alpha-\frac{1}{10^{39}}<\sum_{a\ge 1} \frac{2\log(2^a)}{2^a}=4\log 2,
$$
a contradiction. 

Thus, $t\ge 1$ so $n_1>1$. We now put
$$
{\mathcal I}:=\{i: r_i\mid m\}\quad {\text{\rm and}}\quad {\mathcal J}=\{1,\ldots,t\}\backslash {\mathcal I}.
$$
We put
$$
M=\prod_{i\in {\mathcal I}} r_i.
$$
We also let $j$ be minimal in ${\mathcal J}$. We split the sum appearing in \eqref{eq:veryuseful} in two parts:
$$
\sum_{d\mid n} S_d=L_1+L_2,
$$
where
$$
L_1:=\sum_{\substack{d\mid n\\ r\mid d\Rightarrow r\mid 2M}} S_d\quad {\text{\rm and}}\quad L_2:=\sum_{\substack{d\mid n\\ r_u\mid d~{\text{\rm for~some}}~ u\in {\mathcal J}}} S_d.
$$
To bound $L_1$, we note that all divisors involved divide $n'$, where
$$
n'=2^s\prod_{i\in {\mathcal I}} r_i^{\lambda_i}.
$$
Using Lemmas \ref{lem:sumSn} and \ref{lem:useful}, we get
\begin{eqnarray}
\label{eq:S1}
L_1 & \le & 2\sum_{d\mid n'} \frac{\log d}{d}\nonumber\\
& < & 2\left(\sum_{r\mid n'} \frac{\log r}{r-1}\right)\left(\frac{n'}{\phi(n')}\right)\nonumber\\
& = & 2\left(\sum_{r\mid 2M} \frac{\log r}{r-1}\right) \left(\frac{2M}{\phi(2M)}\right).
\end{eqnarray}
We now bound $L_2$. If ${\mathcal J}=\emptyset$, then $L_2=0$ and there is nothing to bound. So, assume that ${\mathcal J}\ne \emptyset$. We argue as follows. Note that since $s\ge 1$, by Lemma \ref{lem:PQ} (i), we have
$$
P_n=P_{n_1} Q_{n_1} Q_{2n_1}\cdots Q_{2^{s-1} n_1}.
$$
Let $q$ be any odd prime factor of $Q_{n_1}$. By reducing relation (ii) of Lemma \ref{lem:PQ} modulo $q$ and using the fact that $n_1$ and $q$ are both odd, we get
$2P_{n_1}^2\equiv 1\pmod q$, therefore ${\displaystyle{\left(\frac{2}{q}\right)=1}}$. Hence, $z(q)\mid q-1$ for such primes $q$. Now let $d$ be any divisor of $n_1$ which is a multiple of $r_{j}$. The number of them is $\tau(n_1/r_{j})$, where $\tau(u)$ is the number of divisors of the positive integer $u$. For each such $d$, there is a primitive prime factor $q_d$ of $Q_d\mid Q_{n_1}$. Thus, $r_{j}\mid d\mid q_d-1$. This shows that
\begin{equation}
\label{eq:x}
\nu_{r_{j}}(\phi(P_n))\ge \nu_{r_{j}} (\phi(Q_{n_1}))\ge \tau(n_1/r_{j})\ge \tau(n_1)/2,
\end{equation}
where the last inequality follows from the fact that
$$
\frac{\tau(n_1/r_{j})}{\tau(n_1)}=\frac{\lambda_{j}}{\lambda_{j}+1}\ge \frac{1}{2}.
$$
Since $r_{j}$ does not divide $m$, it follows from \eqref{eq:nupvaluationofPn} that
\begin{equation}
\label{eq:xx}
\nu_{r_{j}} (P_m)\le e_{r_{j}}.
\end{equation}
Hence, \eqref{eq:x}, \eqref{eq:xx} and \eqref{eq:pb} imply that
\begin{equation}
\label{eq:taun1}
\tau(n_1)\le 2e_{r_j}.
\end{equation}
Invoking Lemma \ref{lem:zofp}, we get
\begin{equation}
\label{eq:tau}
\tau(n_1)\le \frac{(r_{j}+1)\log \alpha}{\log r_{j}}.
\end{equation}
Now every divisor $d$ participating in $L_2$ is of the form $d=2^a d_1$, where $0\le a\le s$ and $d_1$ is a divisor of $n_1$ divisible by $r_u$ for some $u\in {\mathcal J}$. Thus,
\begin{equation}
\label{eq:tt}
L_2\le \tau(n_1) \min\left\{ \sum_{\substack{0\le a\le s\\  d_1\mid n_1\\ r_u\mid d_1~{\text{\rm for~some}}~ u\in {\mathcal J}}} S_{2^ ad_1}\right\}:=g(n_1,s,r_1).
\end{equation}
In particular, $d_1\ge 3$ and since the function $x\mapsto \log x/x$ is decreasing for $x\ge 3$, we have that
\begin{equation}
\label{eq:S2}
g(n_1,s,r_1)\le 2\tau(n_1) \sum_{0\le a\le s} \frac{\log(2^a r_{j})}{2^a r_{j}}.
\end{equation}
Putting also $s_1:=\min\{s,416\}$, we get, by Lemma \ref{lem:1}, that $2^{s_1}\mid \ell$. Thus, inserting this as well as \eqref{eq:S1} and \eqref{eq:S2} all into \eqref{eq:veryuseful}, we get
\begin{equation}
\label{eq:main2}
\ell \log \alpha-\frac{1}{10^{39}}<2\left(\sum_{r\mid 2M} \frac{\log r}{r-1}\right) \left(\frac{2M}{\phi(2M)}\right)+g(n_1,s,r_1).
\end{equation}
Since
\begin{equation}
\label{eq:3***}
\sum_{0\le a\le s}
\frac{\log(2^a r_{j})}{2^a r_{j}}<\frac{4\log 2+2\log r_{j}}{r_{j}},
\end{equation}
inequalities \eqref{eq:3***}, \eqref{eq:tau} and \eqref{eq:S2} give us that
$$
g(n_1,s,r_1)\le 2\left(1+\frac{1}{r_j}\right)\left(2+\frac{4\log 2}{\log r_j}\right)\log \alpha:=g(r_j).
$$
The function $g(x)$ is decreasing for $x\ge 3$.  Thus, $g(r_j)\le g(3)< 10.64$.
For a positive integer $N$ put
\begin{equation}
\label{eq:f(M)}
f(N):=N\log \alpha-\frac{1}{10^{39}}- 2\left(\sum_{r\mid N} \frac{\log r}{r-1}\right) \left(\frac{N}{\phi(N)}\right).
\end{equation}
Then inequality \eqref{eq:main2} implies that both inequalities
\begin{eqnarray}
\label{eq:200}
f(\ell)<g(r_j),\nonumber\\
(\ell-M)\log \alpha+f(M)<g(r_j)
\end{eqnarray}
hold. Assuming that $\ell\ge 26$, we get, by Lemma \ref{lem:RS}, that
\begin{eqnarray*}
\ell \log \alpha-\frac{1}{10^{39}}-2(\log 2) \frac{(1.79 \log\log \ell+2.5/\log\log \ell)\log \ell}{\log\log \ell-1.1714}\le 10.64.
\end{eqnarray*}
Mathematica confirmed that the above inequality implies $\ell\le 500$. Another calculation with Mathematica
showed that the inequality
\begin{equation}
\label{eq:111}
f(\ell)<10.64
\end{equation}
for even values of $\ell\in [1,500]\cap {\mathbb Z}$ implies that $\ell\in [2,18]$. The minimum of the function
$f(2N)$ for $N\in [1,250]\cap {\mathbb Z}$ is at $N=3$ and $f(6)>-2.12$. For the remaining positive integers $N$, we have $f(2N)>0$. Hence, inequality \eqref{eq:200} implies
$$
(2^{s_1}-2)\log \alpha<10.64\quad {\text{\rm and}}\quad (2^{s_1}-2) 3\log \alpha<10.64+2.12=12.76,
$$
according to whether $M\ne 3$ or $M=3$, and either one of the above inequalities implies that $s_1\le 3$. Thus, $s=s_1\in \{1,2,3\}$.
Since $2M\mid \ell$, $2M$ is square-free and $\ell\le 18$, we have that $M\in \{1,3,5,7\}$. Assume $M>1$ and let $i$ be such that $M=r_i$. Let us show that $\lambda_i=1$. Indeed, if $\lambda_i\ge 2$, then $$
199\mid Q_9\mid P_n,\quad 29201\mid P_{25}\mid P_n,\quad 1471\mid Q_{49}\mid P_n,
$$
according to whether $r_i=3,~5$, $7$, respectively, and $3^2\mid 199-1,~5^2\mid 29201-1,~ 7^2\mid 1471-1$. Thus, we get that
$3^2,~ 5^2,~7^2$ divide $\phi(P_n)=P_m$, showing that $3^2,~5^2,~7^2$ divide $\ell$. Since $\ell\le 18$, only the case $\ell=18$ is possible.  In this case, $r_j\ge 5$, and inequality \eqref{eq:200} gives
$$
8.4<f(18)\le g(5)<7.9,
$$
a contradiction. Let us record what we have deduced so far.

\begin{lemma}
\label{lem:4}
If $n>2$ is even, then $s\in \{1,2,3\}$. Further, if ${\mathcal I}\ne\emptyset$, then ${\mathcal I}=\{i\}$, $r_i\in \{3,5,7\}$ and $\lambda_i=1$.
\end{lemma}

We now deal with ${\mathcal J}$. For this, we return to \eqref{eq:veryuseful} and use the better inequality namely
$$
2^s M\log \alpha-\frac{1}{10^{39}}\le \ell \log \alpha-\frac{1}{10^{39}}\le
\log \left(\frac{P_n}{\phi(P_n)}\right)\le \sum_{d\mid 2^s M} \sum_{p\in {\mathcal P}_d} \log\left(1+\frac{1}{p-1}\right)+L_2,
$$
so
\begin{equation}
\label{eq:300}
L_2\ge 2^s M\log \alpha-\frac{1}{10^{39}}-\sum_{d\mid 2^s M} \sum_{p\in {\mathcal P}_d} \log\left(1+\frac{1}{p-1}\right).
\end{equation}
In the right--hand side above, $M\in \{1,3,5,7\}$ and $s\in \{1,2,3\}$. The values of the right--hand side above are in fact
$$
h(u):=u\log \alpha-\frac{1}{10^{39}} -\log(P_u/\phi(P_u))
$$
for $u=2^sM\in \{2,4,6,8,10,12,14,20,24,28,40,56\}$. Computing we get:
$$
h(u)\ge H_{s,M} \left(\frac{M}{\phi(M)}\right)\quad {\text{\rm for}}\quad M\in \{1,3,5,7\},\quad s\in \{1,2,3\}, 
$$
where
$$
H_{1,1}>1.069,\quad H_{1,M}>2.81\quad {\text{\rm for}}\quad M>1,\quad H_{2,M}>2.426,\quad H_{3,M}>5.8917.
$$
We now exploit the relation
\begin{equation}
\label{eq:L2}
H_{s,M} \left(\frac{M}{\phi(M)}\right)<L_2.
\end{equation}
Our goal is to prove that $r_j<10^6$. Assume this is not so. We use the bound
$$
L_2<\sum_{\substack{d\mid n\\ r_u\mid d~{\text{\rm for~sume}}~u\in {\mathcal J}}} \frac{4+4\log\log d}{\phi(d)}
$$
of Lemma \ref{lem:sumSn}. Each divisor $d$ participating in $L_2$ is of the form $2^a d_1$, where $a\in [0,s]\cap {\mathbb Z}$ and $d_1$ is a multiple of a prime at least as large as $r_j$. Thus,
$$
\frac{4+4\log\log d}{\phi(d)}\le \frac{4+4\log\log 8d_1}{\phi(2^a) \phi(d_1)}\quad {\text{\rm for}}\quad a\in \{0,1,\ldots,s\},
$$
and
$$
\frac{d_1}{\phi(d_1)}\le \frac{n_1}{\phi(n_1)}\le \frac{M}{\phi(M)}\left(1+\frac{1}{r_j-1}\right)^{\omega(n_1)}.
$$
Using \eqref{eq:tau}, we get
$$
2^{\omega(n_1)}\le \tau(n_1)\le \frac{(r_j+1)\log \alpha}{\log r_j}<r_j,
$$
where the last inequality holds because $r_j$ is large. Thus,
\begin{equation}
\label{eq:omega}
\omega(n_1)<\frac{\log r_j}{\log 2}<2\log r_j.
\end{equation}
Hence,
\begin{eqnarray}
\label{eq:totient}
\frac{n_1}{\phi(n_1)}&\le& \frac{M}{\phi(M)}\left(1+\frac{1}{r_j-1}\right)^{\omega(n_1)}<\frac{M}{\phi(M)}\left(1+\frac{1}{r_j-1}\right)^{2\log r_j}\nonumber\\
&<&\frac{M}{\phi(M)}\exp\left(\frac{2\log r_j}{r_j-1}\right)<\frac{M}{\phi(M)}\left(1+\frac{4\log r_j}{r_j-1}\right),
\end{eqnarray}
where we used the inequalities  $1+x<e^x$, valid for all real numbers $x$, as well as $e^x<1+2x$ which is valid for $x\in (0,1/2)$ with $x=2\log r_j/(r_j-1)$ which belongs to $(0,1/2)$ because $r_j$ is large.  Thus, the inequality
$$
\frac{4+4\log\log d}{\phi(d)}\le \left(\frac{4+4\log\log 8d_1}{d_1}\right)\left(1+
\frac{4\log r_j}{r_j-1}\right) \left(\frac{1}{\phi(2^a)}\right) \frac{M}{\phi(M)}
$$
holds for $d=2^a d_1$ participating in $L_2$. The function $x\mapsto (4+4\log\log(8x))/x$ is decreasing for $x\ge 3$. Hence,
\begin{equation}
\label{eq:imp11}
L_2\le \left(\frac{4+4\log\log(8 r_j)}{r_j}\right) \tau(n_1) \left(1+\frac{4\log r_j}{r_j-1}\right)\left(\sum_{0\le a\le s} \frac{1}{\phi(2^a)}\right) \left(\frac{M}{\phi(M)}\right).
\end{equation}
Inserting inequality \eqref{eq:tau} into \eqref{eq:imp11} and using \eqref{eq:L2}, we get
\begin{equation}
\label{eq:r1}
\log r_j<4\left(1+\frac{1}{r_j}\right)\left(1+\frac{4\log r_j}{r_j-1}\right)(1+\log\log (8 r_j))(\log\alpha)\left( \frac{G_s}{H_{s,M}}\right),
\end{equation}
where
$$
G_s=\sum_{0\le a\le s} \frac{1}{\phi(2^a)}.
$$
For $s=2,~3$, inequality \eqref{eq:r1} implies $r_j<900,000$ and $r_j<300$, respectively. For $s=1$ and $M>1$, inequality \eqref{eq:r1}  implies $r_j<5000$. When $M=1$ and $s=1$, we get $n=2n_1$ and $j=1$. Here, inequality \eqref{eq:r1} implies that $r_1<8\times 10^{12}$. This is too big,  so we use the bound
$$
S_d<\frac{2\log d}{d}
$$
of Lemma \ref{lem:sumSn} instead for the divisors $d$ of participating in $L_2$, which in this case are all the divisors of $n$ larger than $2$. We deduce that
$$
1.06<L_2<2\sum_{\substack{d\mid 2n_1\\ d>2}} \frac{\log d}{d}<4\sum_{d_1\mid n_1} \frac{\log d_1}{d_1}.
$$
The last inequality above follows from the fact that all divisors $d>2$ of $n$ are either of the form $d_1$ or $2d_1$ for some divisor $d_1\ge 3$ of $n_1$, and the function $x\mapsto \log x/x$  is decreasing for $x\ge 3$.
Using Lemma \ref{lem:useful} and inequalities \eqref{eq:omega} and \eqref{eq:totient}, we get
\begin{eqnarray*}
1.06 & < & 4\left(\sum_{r\mid n_1} \frac{\log r}{r-1}\right) \left(\frac{n_1}{\phi(n_1)}\right)<\left(\frac{4\log r_1}{r_1-1}\right) \omega(n_1) \left(1+\frac{4\log r_1}{r_1-1}\right)\\
& < & \left(\frac{4\log r_1}{r_1-1}\right)\left(2\log r_1\right)\left(1+\frac{4\log r_1}{r_1-1}\right),
\end{eqnarray*}
which gives $r_1<159$. So, in all cases, $r_j<10^6$. Here, we checked that $e_{r}=1$ for all such $r$ except $r\in \{13,31\}$ for which $e_{r}=2$. If $e_{r_j}=1$, we then get $\tau(n_1/r_j)\le 1$, so $n_1=r_j$. Thus, $n\le 8\cdot 10^6$, in contradiction with Lemma \ref{lem:1}. Assume now that $r_j\in \{13,31\}$. Say $r_j=13$. In this case, $79$ and $599$ divide $Q_{13}$ which divides $P_n$, therefore $13^2\mid (79-1)(599-1)\mid \phi(P_n)=P_m$. Thus, if there is some other prime factor $r'$ of $n_1/13$, then $13r'\mid n_1$, and $Q_{13r'}$ has a primitive prime factor $q\equiv 1\pmod {13 r'}$. In particular, $13\mid q-1$. Thus, $\nu_{13}(\phi(P_n))\ge 3$, showing that $13^3\mid P_m$. Hence, $13\mid m$, therefore $13\mid M$, a contradiction. A similar contradiction is obtained if $r_j=31$ since $Q_{31}$ has two primitive prime factors namely $424577$ and $865087$ so $31\mid M$. 

This finishes the proof.

\section*{Acknowledgments} B. F. thanks OWSD and Sida (Swedish International Development Cooperation Agency) for a scholarship during her Ph.D. studies at Wits.

\end{document}